\documentclass [a4paper, 12pt, reqno]{amsart}
\usepackage[latin1]{inputenc}
\usepackage[margin=25mm, left=35mm, right=35mm]{geometry}
\usepackage[USenglish]{babel} 
\usepackage{amssymb, amsmath, amsthm}
\usepackage{aliascnt}

\usepackage{marvosym}
\usepackage{microtype}
\usepackage{mathrsfs, calligra}
\usepackage[arrow, matrix, curve]{xy}
\usepackage{setspace}
\usepackage{amscd}
\usepackage{ae,aecompl}
\usepackage{xspace}
\usepackage{enumitem}
\usepackage[breaklinks=true, colorlinks=false]{hyperref}

\newlength{\itemLM} 

\setlist{leftmargin=\itemLM}
\setlist[description]{ \centering\bfseries }
\newcommand{\looseitem}[1]{\makebox[\itemLM]{#1}}
\SetLabelAlign{left}{\makebox[\itemLM]{#1}}
\SetLabelAlign{right}{\hspace{-.75\itemLM}\makebox[.75\itemLM]{#1}}

\newcommand{\spar}{\par\smallskip} 
\newcommand{\mpar}{\par\medskip}
\newcommand{\bpar}{\par\bigskip}

\allowdisplaybreaks

\numberwithin{equation}{section}

\numberwithin{equationX}{section}

\addto\extrasUSenglish{}
\addto\extrasUSenglish{}

\theoremstyle{definition}
\newaliascnt{definition}{equationX} \newtheorem{defn}[definition]{Definition} 
\newaliascnt{remark}{equationX} \newtheorem{rem}[remark]{Remark} 
\newaliascnt{example}{equationX}  
\theoremstyle{plain}
\newaliascnt{theorem}{equationX} \newtheorem{thm}[theorem]{Theorem}
\newaliascnt{lemma}{equationX}  \newtheorem{lem}[lemma]{Lemma} 
\newaliascnt{proposition}{equationX} 
\newaliascnt{corollary}{equationX} \newtheorem{cor}[corollary]{Corollary}

\expandafter\let\expandafter\oldproof\csname\string\proof\endcsname \let\oldendproof\endproof

\renewenvironment{proof}{\oldproof[\textit Proof:]}{\oldendproof}
\newenvironment{proofX}[1]{\oldproof[\textit Proof\ #1:]}{\oldendproof}
\newcommand{\textSkoof}{Sketch of the proof} 

\newcommand{\thereqed}[1]{~\hfill\raisebox{-#1}[0mm][0mm]{\qedhere}}

\newcommand{\newterm}{\textit}

\newcommand{\nbd}{\nobreakdash-}

       \newcommand{\cO} {\mathcal{O}}   

\newcommand{\am}{\mathfrak{m}}   

\newcommand{\sC}{\mathscr{C}}   \newcommand{\sE}{\mathscr{E}} \newcommand{\sF}{\mathscr{F}} \newcommand{\sG}{\mathscr{G}}   \newcommand{\sI}{\mathscr{I}} \newcommand{\sK}{\mathscr{K}} \newcommand{\sL}{\mathscr{L}}    \newcommand{\sS}{\mathscr{S}} \newcommand{\sT}{\mathscr{T}}  
\newcommand{\sJ}{\mathscr{J}{\!}}
\DeclareMathOperator{\sKer}{\mathscr{K}\text{\kern -4pt {\calligra\large er}}\,} \DeclareMathOperator{\sIm}{\mathscr{I}\text{\kern -5pt {\calligra\large m}}\,} \DeclareMathOperator{\sHom}{\mathscr{H}\text{\kern -4pt {\calligra\large om}}\,} \DeclareMathOperator{\sExt}{\mathscr{E}\text{\kern -3pt {\calligra xt}}\,} \DeclareMathOperator{\sCoker}{\mathscr{C}\text{\kern -3pt {\calligra oker}}\,}

  \newcommand{\CC}{\mathbb{C}} \newcommand{\CP}{\mathbb{CP}}      

 \newcommand{\codim}{\mathrm{codim}}      \newcommand{\Hom}{\mathrm{Hom}} 
  \newcommand{\pr}{\mathrm{pr}}  \newcommand{\reg}{\mathrm{reg}}   \newcommand{\sing}{\mathrm{sing}}   

       \DeclareMathOperator{\PC}{PC}     \DeclareMathOperator{\red}{red}  \DeclareMathOperator{\Sing}{Sing}      \DeclareMathOperator{\Specan}{Specan}
\DeclareMathOperator{\supp}{supp}  
\DeclareMathOperator{\rk}{rk}  \DeclareMathOperator{\rank}{rk} \DeclareMathOperator{\corank}{cork} \DeclareMathOperator{\cork}{cork}

\newcommand{\cf}{cf.\ }
\newcommand{\eg}{e.\,g.\ }  
\newcommand{\ie}{i.\,e., } \newcommand{\Ie}{I.\,e., }

\newcommand{\aus}{\subset}    \newcommand{\minus}{{\setminus}}   \newcommand{\stimes}{{\times}}  \newcommand{\tensor}{\otimes}\newcommand{\stensor}{{\otimes}} 

\newcommand{\abb}{\rightarrow} \newcommand{\auf}{\mapsto} \newcommand{\iso}{\cong}
\newcommand{\nach}{\circ} 

\newcommand{\kl}{\leq} \newcommand{\gr}{\geq} \newcommand{\ungl}{\neq}

\newcommand{\ph}{\varphi}

\newcommand{\Thm}{Thm.~}
\newcommand{\Lem}{Lem.~}
\newcommand{\Prop}{Prop.~}
\newcommand{\Sec}{Sect.~}
\newcommand{\Cor}{Cor.~}
\newcommand{\Chap}{Chap.~}

\renewcommand{\tilde}{\widetilde}\renewcommand{\hat}{\widehat}

\newlength{\FErgLaenge}

\newcommand{\FErgT}[1]{\settowidth{\FErgLaenge}{#1} \hspace{5mm}\hbox{#1}\hspace{-\FErgLaenge}}

\newcommand{\spimply}[2]{(#1)\,$\Rightarrow$\,(#2)}

\makeatletter
\def\section{\@startsection{section}{1}%
  \z@{2\linespacing\@plus1.5\linespacing}{\linespacing\@plus .5\linespacing }
  {\normalfont\bf\centering}}

\makeatother

\hypersetup{
  pdftoolbar=true,
  pdfmenubar=true,
  pdftitle={Modifications of torsion-free coherent analytic sheaves},
  pdfauthor={Jean Ruppenthal and Martin Sera},
  pdfsubject={Preprint},
  pdfkeywords={Modifications, Coherent Analytic Sheave, Torsion-free Preimage},
pdfborder= 0 0 .1,
}

\title[Modifications of torsion-free sheaves]{Modifications of torsion-free\\ coherent analytic sheaves}

\author[J. Ruppenthal and M. Sera]{Jean Ruppenthal and Martin Sera}
\address{Department of Mathematics and Informatics, University of Wuppertal, Gau{\ss}\-str.~20, 42119 Wuppertal, Germany.}
\email{\href{mailto:ruppenthal@uni-wuppertal.de}{ruppenthal@uni-wuppertal.de}, \href{mailto:sera@math.uni-wuppertal.de}{sera@math.uni-wuppertal.de}}

\date{\today}


\begin{document}

\begin{abstract}
We study the transformation of torsion-free coherent analytic \linebreak
sheaves under proper modifications.
More precisely, we study direct images of inverse image sheaves, and torsion-free preimages of direct image sheaves.
Under some conditions, it is shown that torsion-free coherent sheaves can be realized as the direct image of locally free sheaves under modifications. 
Thus, it is possible to study coherent sheaves modulo torsion by reducing the problem to study vector bundles on manifolds.
We apply this to reduced ideal sheaves and to the Grauert-Riemenschneider canonical sheaf of holomorphic $n$-forms.
\end{abstract}

\maketitle

\section{Introduction}

In bimeromorphic geometry, the use of locally free coherent analytic sheaves is limited:
The direct image of a locally free sheaf under a proper modification is not locally free any more.
Instead, it is reasonable to consider the wider category of torsion-free coherent analytic sheaves.
The restriction to torsion-free sheaves makes sense for bimeromorphic considerations as the torsion
of a coherent analytic sheaf is supported on analytically thin subsets.
To exemplify the use of torsion-free sheaves, just recall that an irreducible\footnote{We call a complex space \newterm{locally irreducible} if each stalk of its structure sheaf is an integral domain (see \eg \S 1.5 in \cite[\Chap 1]{GrauertRemmertCAS}). In particular, it is then automatically reduced. On the other hand, we say that a complex space $X$ is \newterm{(globally) irreducible} if the underlying reduced space, $\red(X)$,
consists of just one irreducible component. If there are more than one components,
then $X$ is called \newterm{reducible}.}
reduced compact space $X$ is a Moishezon
space if and only if it carries a positive torsion-free coherent analytic sheaf $\sS$ with $\supp(\sS)=X$
(see \eg \Thm 6.14 in \cite{Peternell94}).

\medskip
Let $\pi\colon Y \rightarrow X$ be a proper modification of a complex space $X$ (cf. \autoref{sec:prelims} for notations).
Then the direct image $\pi_* \sF$ of a torsion-free coherent analytic sheaf $\sF$ remains torsion-free.
The problem is here that the analytic inverse image sheaf $\pi^* \sS$ of a torsion-free coherent analytic sheaf $\sS$
is not torsion-free in general. For a counter-example, see \eg the example in \cite[\Sec 1]{GrauertRiemenschneider70},
\ie the pullback of the maximal ideal sheaf of the origin in $\CC^2$ under blow-up of the origin is not torsion-free.
One can say more or less that $\pi^* \sS$ is torsion-free in a point $y\in Y$ if and only if $\sS$ is locally free in $\pi(y)$
(see \cite{Rabinowitz79} or \autoref{rem:Rabinowitz} below).

\medskip
To take care of the torsion which arises when taking analytic inverse images,
it is useful to consider the torsion-free preimage sheaf:

\begin{defn}
Let $\pi\colon Y\rightarrow X$ be a holomorphic map between complex spaces such that $Y$ is locally irreducible.
Let $\sS$ be a coherent analytic sheaf on $X$. Then
\[\pi^T \sS := \pi^* \sS / \sT(\pi^*\sS),\]
where $\sT(\pi^*\sS)$ is the torsion sheaf of $\pi^*\sS$, is called the \newterm{torsion-free preimage sheaf} of $\sS$
under $\pi$.
\end{defn}

Torsion-free preimages under proper modifications have been first studied by H.~Rossi \cite{Rossi68},
H.~Grauert and O.~Riemenschneider \cite{GrauertRiemenschneider70, Riemenschneider71}.
The main motivation is as follows: Let $\sS$ be a torsion-free coherent analytic sheaf on an irreducible
complex space $X$. Then Rossi showed that there exists a proper modification $\ph_\sS\colon Y \rightarrow X$
such that $\ph_\sS^T \sS$ is locally free (see \autoref{sec:prelims} for the details).
Combining this with a resolution of singularities $\sigma\colon M\rightarrow Y$ which exists due to Hironaka,
we obtain a resolution of singularities $\pi=\ph_\sS\circ\sigma\colon M\rightarrow X$ such that $\pi^T \sS$
is locally free. Thus, it is possible to study coherent analytic sheaves modulo torsion by reducing the problem
to study vector bundles on manifolds.

\medskip
In view of this idea, it seems very interesting to study the connection between $\sS$ and its torsion-free
preimage $\pi^T \sS$ closer, and we have found the following relation which to our knowledge has not been observed in the literature:

\begin{thm}\label{thm:mainA} 
Let $\pi\colon Y \rightarrow X$ be a proper modification of a complex space $X$,
and let $\sF$ and $\sG$ be torsion-free coherent analytic sheaves on $X$ and $Y$, respectively.
\begin{itemize}
\item[(i)] If $\sF=\pi_\ast \sG$, then
	\[\sF \iso \pi_*\pi^T \sF.\]
\item[(ii)] If $\sG = \pi^T\sF$, then
	\[\pi^T\pi_\ast \sG\iso \sG.\]
\end{itemize}
\end{thm}
\spar
\autoref{thm:mainA} can be shown directly by standard facts on modifications and torsion-free sheaves.
But we obtain it here as a simple byproduct of considerations on the functor $\pi^T$  (see \autoref{lem:quasiexactness}) and the fact that the natural maps 
\begin{eqnarray*}
\sF &\abb& \pi_\ast\pi^T \sF,\\
\pi^T\pi_\ast\sG &\abb& \sG
\end{eqnarray*}
both are injective (\autoref{thm:Injection1+MainNormal} (i) and \autoref{lem:Injection2}; 
the proof of \autoref{lem:Injection2} presented here is due to Matei Toma).
We give also counter-examples to show
that these injections are not bijective in general (\autoref{rem:CE1} and \autoref{rem:CE2}).

\medskip
In order to understand relations as in \autoref{thm:mainA} better,
it turns out useful to study properties of linear spaces associated to coherent analytic sheaves.
Among other things, we obtain the following equivalence:
\begin{thm}\label{thm:MainIrred}
Let $X$ be a connected factorial Cohen-Macaulay space and $\sS$ a coherent analytic sheaf on $X$ generated by $\rank\sS+m$ sections, $m\,{\kl}\,2$, 
such that the singular locus of $\sS$ is at least of codimension $m{+}1$ in $X$.
Then the following conditions are equivalent:
\begin{itemize}
\item[(1)] $\sS$ is torsion-free.
\item[(2)] The linear (fiber) space $L(\sS)$ associated to $\sS$ is globally irreducible and reduced
(\ie it consists only of its primary component).
\item[(3)] $L(\sS)$ is locally irreducible.%
\item[(4)] For all $p\in X$, there is a neighborhood $U\aus X$ such that
	\[0 \abb \cO_U^m \abb \cO_U^{\rank \sS+m}\abb\sS_U  \abb 0\]
is exact, \ie the homological dimension of $\sS$ is at most one.
\end{itemize}
\end{thm}

R. Axelsson and J. Magn\'usson proved in \cite[\Prop 3.11]{Axelsson-Magnusson98} that a condition on the codimension of the singular locus of $\sS$ is necessary to obtain the irreducibility of $L(\sS)$. In this context,
we can highly recommend their papers \cite{Axelsson-Magnusson86} and \cite{Axelsson-Magnusson98},
which deal with complex analytic cones, a generalization of complex linear spaces.

\spar
If, moreover, the singular locus of $\sS$ is at least $(m{+}2)$-codimensional, then $L(\sS)$ is normal (see \autoref{thm:NormalLS}). This is a particularly interesting situation because of the following statement.

\begin{thm}\label{thm:MainNormal} 
Let $X$ be a locally irreducible complex space, $\sS$ a torsion-free coherent analytic sheaf on $X$ such that the linear space associated to $\sS$ is normal,
and $\pi\colon Y\abb X$ a proper modification of $X$. Then the canonical homomorphism $\sS \rightarrow \pi_*(\pi^T \sS)$ is bijective, \ie
	\[\sS \iso \pi_\ast (\pi^T \sS).\]
\end{thm}

In this situation, $\sS$ can actually be represented as the direct image of a locally free sheaf.

\medskip
Let us mention two applications of \autoref{thm:mainA}. First, we will study ideal sheaves.
Let $\pi\colon Y \rightarrow X$ be a proper modification of a locally irreducible complex space $X$,
$A\subset X$ an analytic subset with ideal sheaf $\sJ_A$,
and $B:=\pi^{-1}(A)$ the analytic preimage with ideal sheaf $\sJ_B$. Then
$\pi^T \sJ_A = \sJ_B$ (\cf \autoref{lem:idealsheaves1}),
and \autoref{thm:mainA} yields that $\sJ_B \cong \pi^T \pi_* \sJ_B$.

If we assume moreover that $X$ is normal and that $A$ is either a locally complete intersection or a normal analytic set
and that $\sigma\colon Y \rightarrow X$ is the monoidal transformation with respect to $\sJ_A$,
then $\sJ_B=\sigma^T \sJ_A$ (is locally free) and we have (see \autoref{lem:idealsheaves}): 
\[\sJ_A \cong \sigma_* \sJ_B \cong \sigma_* \sigma^T \sJ_A.\]

\medskip
Second, let $X$ be a locally irreducible complex space of pure dimension $n$,
and $\sK_X$ the Grauert-Riemenschneider canonical sheaf on $X$ (as introduced in \cite{GrauertRiemenschneider70}).
Then there exists a resolution of singularities $\pi\colon M\rightarrow X$ (with only normal crossings)
such that $\pi^T\sK_X$ is locally free,
and so there is an effective divisor $D$ with support on the exceptional set of the modification such that
\begin{eqnarray}\label{eq:intapp1}
\sK_X \cong \pi_* \pi^T \sK_X = \pi_* \Omega^n_M(-D) = \pi_* \big( \Omega^n_M\otimes \cO(-D) \big)
\end{eqnarray}
(see \autoref{thm:nforms}). Let us explain briefly the meaning of \eqref{eq:intapp1}.
By definition of the Grauert-Riemenschneider canonical sheaf, we know already that $\sK_X \cong \pi_* \Omega^n_M$.
\eqref{eq:intapp1} tells us that we can as well consider the push-forward of holomorphic $n$-forms
which vanish to the order of $D$ on the exceptional set.
This is a useful information, particularly if $\pi$ is explicitly given so that $D$ can be calculated explicitly.
An example:
If $X$ is already a manifold (\ie $\sK_X=\Omega^n_X$) 
and $\pi\colon M \rightarrow X$ is the blow-up along a submanifold of codimension $s$ in $X$
with exceptional set $E$, then (see \eg \Prop VII.12.7 in \cite{DemaillyAG}):
\[\pi^T \sK_X = \pi^* \sK_X = \Omega^n_M\big(-(s-1)E\big),\] 
and so 
\[\Omega^n_X = \sK_X \cong \pi_* \Omega^n_M\big(-(s-1)E\big).\]
Considerations of this kind are particularly important in the study of canonical sheaves on singular complex spaces
(see \cite{Ruppenthal14}). We will set up the relation \eqref{eq:intapp1} also for holomorphic $n$-forms with values
in locally free coherent analytic sheaves (see \autoref{thm:app1}).

\medskip
Using \autoref{thm:MainNormal}, we are able to generalize Takegoshi's relative version \cite{Take85}
of the Grauert-Riemenschneider vanishing theorem in several directions. 
This is elaborated by the second author in \cite{Sera16}.

\medskip
The content of the present paper is organized as follows.
After a brief review of monoidal transformations with respect to coherent analytic sheaves in \autoref{sec:prelims},
we will study linear spaces associated to torsion-free coherent analytic sheaves in \autoref{sec:LinearSpaces}. There, we will prove  \autoref{thm:MainIrred}.
Then we study direct images of (torsion-free) analytic preimage sheaves (including the proof of \autoref{thm:MainNormal}) in \autoref{sec:Injection1},
and torsion-free analytic preimages of direct image sheaves in \autoref{sec:Injection2}.
In \autoref{sec:mainA}, we show that the analytic inverse image functor preserves monomorphisms and epimorphisms
and use this fact in combination with the previous considerations to prove \autoref{thm:mainA}.
\autoref{sec:application1} and \autoref{sec:hnforms} contain the applications described above.
We complement the paper by analogous considerations on the non-analytic inverse image functor in \autoref{sec:NAInverseImage}.

\bigskip
{\bf Acknowledgments.}
{The authors are grateful to Matei Toma and Daniel Greb for discovering a mistake in an earlier version of this paper
and for helpful discussions on the topic,
and to the unknown referee for several suggestions which helped to improve the paper.
This research was supported by the Deutsche Forschungsgemeinschaft (DFG, German Research Foundation), 
grant RU 1474/2 within DFG's Emmy Noether Programme.}

\section{Monoidal transformations}
\label{sec:prelims}

Let us recall some preliminaries on monoidal transformations of complex spaces with respect to coherent analytic sheaves.

\begin{defn}
A proper surjective holomorphic map $\ph\colon X \rightarrow Y$ of complex spaces $X$ and $Y$ is called a
{\it (proper) modification} if there are closed analytic sets $A\subset X$ and $B\subset Y$ such that
\begin{itemize}
\item[(1)] $B=\ph(A)$, 
\item[(2)] $\ph|_{X\setminus A}\colon X\setminus A \rightarrow Y\setminus B$ is biholomorphic,
\item[(3)] $A$ and $B$ are analytically rare, and
\item[(4)] $A$ and $B$ are minimal with the properties (1\,--\,3).
\end{itemize}

$A$ is called the \newterm{exceptional set of $\ph$} and $B$ the \newterm{center} of the modification.
\end{defn}

Rossi showed in \cite{Rossi68} that coherent analytic sheaves can be made locally free by use of modifications.
This process has been treated more systematically by Riemenschneider \cite{Riemenschneider71}.
Following \cite[\S\,2]{Riemenschneider71}, we define:

\begin{defn}
Let $X$ be a complex space and $\sS$ a coherent analytic sheaf on $X$.
Then a pair $(X_\sS,\varphi_\sS)$ of a complex space $X_\sS$ and a proper modification
$\varphi_\sS\colon X_\sS \rightarrow X$ is called the {\it monoidal transformation of $X$ with respect to $\sS$}
if the following two conditions are fulfilled:
\begin{itemize}
\item[(1)] the torsion-free preimage $\varphi_\sS^T \sS=\varphi_\sS^* \sS/ \sT(\varphi_\sS^*\sS)$ is locally free on $X_\sS$,
\item[(2)] if $\pi\colon Y \rightarrow X$ is any proper modification such that (1) holds then there is a unique
holomorphic mapping $\psi\colon Y \rightarrow X_\sS$ such that $\pi= \varphi_\sS\circ\psi$.
\end{itemize}
\end{defn}

So, if $X_\sS$ exists, it is uniquely determined up to biholomorphism by (2).
But existence was first proven by Rossi (see \Thm 3.5 in \cite{Rossi68})
and then studied further by Riemenschneider (see \Thm 2 in \cite{Riemenschneider71}):

\begin{thm}\label{thm:rossi}
Let $X$ be an (irreducible) complex space, $\sS$ a coherent analytic sheaf on $X$
and $A=\Sing\sS:=\{x\in X\colon \sS \mbox{ is not locally free at } x\}$ the singular locus of $\sS$. Then there exists the monoidal transformation $(X_\sS,\varphi_\sS)$
of $X$ with respect to $\sS$. $X_\sS$ is a reduced (irreducible) complex space and $\varphi_\sS$
is a projective proper modification such that
\[\varphi_\sS\colon X_\sS\setminus \varphi_\sS^{-1}(A) \rightarrow X \setminus A\]
is biholomorphic.
If $U\subset X$ is an open subset, then $(\varphi_\sS^{-1}(U),\varphi_\sS)$ is the monoidal transformation
of $U$ with respect to $\sS_U$.
\end{thm}

\section{Linear spaces of torsion-free coherent analytic sheaves}
\label{sec:LinearSpaces}

For a coherent analytic sheaf $\sS$ on a complex space $X$, 
we work with the \newterm{linear (fiber) space} $S:=L(\sS)$ associated to $\sS$ (in the sense of Fischer \cite{Fischer66, Fischer67} and Grothendieck \cite{Grothendieck61}).
Note that if $\sE$ is a locally free sheaf, then the linear space $L(\sE)$ is the dual of the vector bundle 
which has $\sE$ as sheaf of sections.

Linear fiber spaces are a special case of \newterm{complex analytic cones}, introduced by
R. Axelsson and J. Magn\'usson in \cite{Axelsson-Magnusson86}.
Their further study of complex analytic cones in \cite{Axelsson-Magnusson98}
offers particularly also new, clarifying insights into the theory of linear spaces.

\medskip
\subsection{Primary component of a linear space}
In the following, we will always assume that $X$ is a locally irreducible complex space.
Thus, $X$ decomposes into disjoint connected components which can be considered separately 
(see e.g. \S\,2 in \cite[\Chap 9]{GrauertRemmertCAS}). So, 
we can assume that $X$ is connected, thus also globally irreducible.
For a coherent analytic sheaf $\sS$, 
let $A$ be the thin analytic set in $X$ where $\sS$ is not locally free (see \cite{Grauert62}). We call $\Sing S = \Sing\sS:= A$ the \newterm{singular locus} of $S=L(\sS)$ and $\sS$, respectively.
As $X$ is irreducible, $X':=X\minus (A\cup X_\sing)$ and $A^c:=X\minus A$ are connected. 
$S_U\iso U\stimes\CC^r$, for small open sets $U\aus A^c$, implies that $S_{A^c}$ is also connected. 
The set $S_A$ is an analytic subset of $S$. 
Let $E$ be the irreducible component of $\red(S)$ which contains $S_{A^c}$.
$\PC(S):=E$ is called the {\it primary component} of $S$ (following the notation of \cite{Rabinowitz79}).
We have the decomposition $S=E\cup S_A$.

\begin{rem}\label{rem:CharTFWithPC}Let $S$ be a linear space associated to a coherent analytic sheaf $\sS$. Let $s\in \Hom(S_U, U\stimes\CC)\iso \sS(U)$ be a section. Then the primary component $E$ of $S$ determines $s$ up to torsion, \ie
if $s|_E=0$, then $s\in \sT(\sS)$. This is clear as $s|_E=0$ implies that $s$ is supported only on an analytically thin set.
\end{rem}

\begin{lem} \label{lem:TorsionImpliesRed}
Let $X$ be a locally irreducible complex space and $\sS$ a coherent analytic sheaf on $X$. 
Let $S=L(\sS)$ be the linear space associated to $\sS$ and $E$ its primary component.
If $\sS$ has a torsion element, then $E\neq S$. In particular, $S$ is reducible or non-reduced.
\end{lem}

This observation is an immediate consequence of \Thm 3.10 in \cite{Axelsson-Magnusson98}.
Since it can also be deduced easily from the definition of the primary component, we would like to include a proof here:

\begin{proof} Assume that $\sS$ has a torsion element, \ie there is an open set $U\aus X$, an $s\in\sS(U)\iso\Hom(S_U,U\stimes \CC)$ (see \cite{Fischer67}) and an $r\in\cO_X(U)$ such that $s, r\ungl 0$ but $r\cdot s=0$ on $S_U$. As $X$ is locally irreducible, we can assume that $U$ is irreducible. So, there is a dense open set $V\aus U$ such that $r\in\cO_X^\ast(V)$. Thus $s|_{S_V}= 0$. But $V\cap (\Sing S)^c$ is also open and dense in $U$. So, $s|_E=0$ by the identity theorem as $E$ is irreducible. Since $s\ungl 0$, $S_U$ has to contain (parts of) other irreducible components than $E$ or is not reduced.\end{proof}

Obviously, a linear space can be reducible but reduced (for instance, $\{xz=0\}\aus\CC_x\stimes\CC_z$). The case that a linear space is irreducible and non-reduced can occur as well:
Let $S\aus \CC_{x,y}^2\times\CC_{z,w}^2$ be the linear space given by the ideal sheaf generated by $h_1(x,y;z,w):=y^2z-xyw$ and $h_2(x,y;z,w):=xy z-x^2w$. Then the primary component $E:=\PC(S)$ is given by $g(x,y;z,w):=yz-xw$, \ie $E=\red S$ and $S$ is irreducible. Yet, $g$ does not vanish on $S$ (in the unreduced sense) while $g^2=yzg-xwg=zh_1-wh_2$ does.

\begin{rem}\label{rem:CEIrred} The converse of \autoref{lem:TorsionImpliesRed} is not true:
 
Let $\sJ$ be the ideal sheaf generated by $x^2, xy, y^2$ on $\CC^2_{x,y}$ and $S:=L(\sJ)$ the linear space associated to $\sJ$. Since $\sJ$ can not be generated by 2 elements, we get $\rank S_0 =\rank \sJ_0 = 3$. Hence, $S_0$ is a 3-dimensional analytic subset of $S$. On the other hand, the primary component has dimension $2+\rank S=3$. Hence, $S$ is not irreducible. More precisely,
$S$ is given in $\CC^2_{x,y}\stimes\CC^3_{z}$
by the ideal sheaf generated by $h_1(x,y; z):= y z_1 - x z_2$ and $h_2(x,y; z):= yz_2 - x z_3$ where $z=(z_1,z_2,z_3)$. Since $y(z_2^2 - z_1 z_3) = z_2 h_2$,  the primary component $\PC(S)$ is defined by the functions $h_1, h_2$ and $z_2^2 - z_1 z_3$.
This shows also that the fibers of $\PC(S)$ are not linear (the fiber over the origin is just $\{z_2^2=z_1z_3\}$).
So, the primary component is in general not a linear space, neither in the sense of Fischer \cite{Fischer66, Fischer67}
nor in the sense of Grauert \cite{Grauert62}.%
\footnote{In contrast to Fischer's notion of a linear space,
where it is required that $+\colon S \times_X S \rightarrow S$ is a holomorphic map,
Grauert requires in \cite{Grauert62} (only) that the addition $+\colon S \oplus_X S \rightarrow S$ is holomorphic.
That gives a different
category of linear spaces (which is no longer dually equivalent to the category of coherent analytic sheaves).
In \cite{Rabinowitz78}, page 238, Rabinowitz claims that the primary component of a linear space
is a linear space in the sense of Grauert \cite{Grauert62}, but not in the sense of Fischer \cite{Fischer66, Fischer67}.
Our example shows that even this is not the case. More details and criteria for a Grauert linear space to be linear in Fischer's sense can be found in \cite[\Sec
 3.3]{Axelsson-Magnusson98}.}

Actually, we see that $S$ is the linear space associated to a torsion-free sheaf,
and it is reduced but not irreducible.
Considering analogously the ideal sheaf given by $x^2, xy^2, y^4$,
then it turns out that the associated linear space is neither irreducible nor reduced.

\spar
In general, for an ideal sheaf $\sI$ on a reduced complex space $X$,
the primary component of the associated linear space is given by the analytic spectrum of the Rees algebra of $\sI$ (see Exp.~3.12 (1) in \cite{Axelsson-Magnusson98}):
	\[\PC(L (\sI)) \iso \Specan\left(\bigoplus{\!\big.}_{m\gr 0}\,\sI^m\right).\]
In particular, $S(\sI)\iso\bigoplus_{m\gr 0}\sI^m$ implies that $L(\sI)\iso \PC(L(\sI))$ (\eg if $\sI$ is generated by an $\cO_X$-regular sequence or by two elements, see \cite[\Sec 3]{Barshay73} and \cite[(8) and \Thm 3.1]{Huneke80} resp.).
\end{rem}

\medskip
Using Rossi's monoidal transformation, we can make the following observation about the primary component:

\begin{thm}\label{thm:LocallyIrredOfPC} 
Let $X$ be a locally irreducible complex space and $\sS$ a coherent analytic sheaf on $X$. 
Then the primary component $E$ of the linear space $S$ associated to $\sS$ is locally irreducible.
\end{thm}

\begin{proof}
As above, we can assume that $X$ is connected, \ie irreducible. 
Let
\[\ph:=\ph_\sS\colon X_\sS \rightarrow X\]
be the monoidal transformation of $X$ with respect to $\sS$.
This implies that $\ph$ is biholomorphic on $X_\sS\setminus \ph^{-1}(A)$,
where $A\subset X$ is the thin analytic subset where $\sS$ is not locally free.
Then
\[\ph^* S = X_\sS \times_X S\]
is the linear space associated to $\ph^* \sS$, and there is a proper holomorphic projection
\[\pr\colon \ph^* S \rightarrow S.\]
Now consider the natural surjective homomorphism
\[\ph^* \sS \longrightarrow \ph^T \sS = \ph^* \sS / \sT(\ph^*\sS)\]
which induces a closed embedding of the linear space $V:=L(\ph^T \sS)$ into $\ph^* S$. 
Note that $V$ coincides with $\ph^* S$ on $X_\sS\setminus\ph^{-1}(A)$.
Thus, the vector bundle $V$ is just the primary component of $\ph^* S$,
and $V$ is clearly locally and globally irreducible (because the base space $X_\sS$ is connected and locally irreducible).

\medskip
As $\pr$ is a proper holomorphic mapping, we have that $\pr(V)$ is an irreducible analytic subset of $S$
by Remmert's proper mapping theorem 
and the fact that holomorphic images of irreducible sets are again irreducible
(see \S\,1.3 in \cite[\Chap 9]{GrauertRemmertCAS}).
But $\pr(V)$ coincides with $E$, the primary component of $S$, over $X\setminus A$. Thus:
\begin{equation}\label{eq:praeSurj}\pr(V) = E,\end{equation}
and so $\pr|_V\colon V \rightarrow E$ is a proper modification. Using this and the fact that $V$ is clearly locally irreducible,
it is easy to see that $E$ is also locally irreducible:
For an open connected set $W\subset E$, $\pr|_V^{-1}(W) \subset V$ is again open and connected, thus irreducible since $\ph$ is a proper modification of the irreducible $X$.
But then $W=\pr|_V\big(\pr|_V^{-1}(W)\big)$ is also irreducible by the same argument as above
(holomorphic images of irreducible sets are irreducible).
\end{proof}

\begin{lem}\label{lem:IrredIfPCIsLS} Let $X$ be a locally irreducible complex space and $\sS$ a torsion-free coherent analytic sheaf on $X$.
Then $S=L(\sS)$ is locally irreducible if and only if the primary component of $S$ is a linear space.\end{lem}

\begin{proof} Let $E\aus S$ denote the primary component of $S$. 
As $\sS$ is torsion-free, \cite[\Thm 3.10]{Axelsson-Magnusson98} implies that $E=S$
if and only if the primary component $E$ is a linear space.
Alternatively, this assertion can be proven by elementary computations with \cite[\Lem 1]{Fischer67}:
Assume first that $E$ is a linear space. For all points in $X$, there is a neighborhood $U\aus X$ such that  $E_U$ and $S_U$ are linear spaces in $U\stimes \CC^N$. \Lem 1 in \cite{Fischer67} implies that $E$ is defined by holomorphic functions $h_1,...,h_m\in \cO(U\stimes\CC^N)$ that are fiberwise linear. The restriction of $h_j$ to $S$ gives a section in $\Hom(S_U,U\stimes\CC)\iso\sS(U)$. Since $h_j$ vanishes on the primary component $E$ of $S$, we get $h_j\in \sT(\sS)=0$ (see \autoref{rem:CharTFWithPC}), \ie $S\aus E$. This shows that actually $E=S$. The converse of the assertion is trivial.

The statement of the lemma follows now with \autoref{thm:LocallyIrredOfPC}.
\end{proof}

\spar
As we have seen in the counter-example \autoref{rem:CEIrred}, the primary component need not be a linear space (even in the sense of Grauert). Though, it appears as the
analytic spectrum of a connected graded $\cO_X$-algebra of finite presentation.
So, it is locally isomorphic to a subcone of a trivial linear space defined by quasihomogeneous polynomials (see \Cor 1.13 in \cite{Axelsson-Magnusson86}).
Since a symmetric algebra and its torsion-free reduction is generated by elements of degree one, R. Axelsson and J. Magn{\'u}sson actually
showed in the proof of \cite[\Cor 1.13]{Axelsson-Magnusson86} also the following:

\begin{lem}\label{lem:Homogeneous}
Let $X$ be a (locally irreducible) complex space and $\sS$ a coherent analytic sheaf on $X$. 
Then the primary component $E$ of the linear space $S=L(\sS)$ associated to $\sS$ is fiberwise homogeneous and $E$ is locally defined as analytic set in $U\stimes\CC^N$ by holomorphic fiberwise homogeneous functions for $U\aus X$ small enough.
\end{lem}

We would like to present an alternative proof which uses an argument of G. Fischer instead of \cite[\Cor 1.13]{Axelsson-Magnusson86}:
\begin{proof} As above, let $A\subset X$ be the singular locus of $\sS$. So, $E$ and $S$ coincide over $X\setminus A$, and we only have to show that the fibers of $E$ are homogeneous over points of $A$.

\medskip
The question is local, so consider a point $p\in A$ and a Stein neighborhood $U$ of $p$ in $X$ such that $S_U$ can be realized as a closed linear subspace of $U\stimes \CC^N$. Now $E_U \subset U\stimes\CC^N$ is a closed component of $S_U$ which is linear in the second component over $U\setminus A$. Let $f_1, ..., f_k$ be a set of defining functions for $E_U$ in $U\stimes \CC^N$ ($U$ is chosen to be Stein). For $f_j$, $j=1, ..., k$, we define
	\[f_j^{\hbox{\large$\cdot$}}(\lambda,x,z):= f_j(x,\lambda\cdot z)\]
 on $\CC \stimes U\stimes \CC^N$. Since $E_{U\setminus A}$ is linear, the $f_j^{\hbox{\large$\cdot$}}$ vanish on $(\CC\stimes E)_{U\setminus A}$ and on its closure. By definition, the closure is the irreducible set $\CC\stimes E_U$. Hence, the fibers of $E_U$ are homogeneous. Therefore, $\cdot\colon\CC\stimes E_U\abb  E_U$ given by the restriction of $\cdot:\CC \stimes U \stimes\CC^N\abb U\stimes\CC^N$ is a holomorphic map. Using this, we obtain that the ideal sheaf defining $E_U$ (as analytic set) is generated by fiberwise homogeneous functions (see the first step in the proof of \Lem 1 in \cite{Fischer67}). \end{proof}

If $E {\times_{\!X}} E$ is locally irreducible, one can prove in the same way
that the primary component $E$ is a linear space. Yet, for an irreducible fiber space $E\abb X$, the fiber product of $E{\times_X} E$ need not be reduced (not to mention locally irreducible; for a counter-example, see \Sec 4 in \cite{Fischer66}). Therefore, the restriction of the addition need not be holomorphic.

\mpar
\subsection{Linear spaces of small corank -- Proof of \autoref{thm:MainIrred}}

For a linear space to be (locally) irreducible, it is necessary that the associated coherent analytic sheaf is torsion-free.
In the following, we will prove that this is a sufficient criterion under certain additional assumptions.

\begin{defn}
For a coherent analytic sheaf $\sS$ on a complex space $X$, we define the \newterm{corank of $\sS$} in a point $p\in X$, $\cork_p\sS$,
as the difference of the minimal number of generators of $\sS_p$ and the rank of $\sS$,
and the global corank $\corank \sS:= \sup_{p\in X} \corank_p \sS$.
The corank of a linear space is defined as the corank of the associated coherent analytic sheaf.\end{defn}

\spar
We get the following relation between the corank and the homological dimension of a coherent analytic sheaf:

\begin{lem}\label{lem:corank-criterion-homological-dimension}
Let $X$ be a complex space and $\sS$ a coherent analytic sheaf on $X$. Then, for all $p\in X$, the following is equivalent:
\begin{itemize}
	\item[(1)] There exists a neighborhood $U$ of $p$ such that the following sequence is exact:
		\[\cO_U^{\cork_p \sS}\overset{\alpha}\longrightarrow \cO_U^{\rank \sS+\cork_p\sS} \abb \sS_U\abb 0.\]
	\item[(2)] The homological dimension of $\sS$ in $p$ is less or equal 1, \ie (by definition) there exists a neighborhood $U$ of $p$ such that
		\[0 \abb \cO_U^m \overset{\alpha}\abb \cO_U^{N}\abb\sS_U  \abb 0\]
		is exact for suitable $m$ and $N$.
\end{itemize}
If (1) and (2) are fulfilled, then $\alpha$ in (1) is injective, i.e., 
$m$ and $N$ in (2) can be chosen to be $m=\cork_p\sS$ and $N=\rk\sS+m$.
\end{lem}

\begin{proof} For the implication \spimply12, we just need to show that $\alpha$ is injective: In points where $\sS$ is locally free, $\alpha$ is injective (due to the rank\,/\,dimension). Hence, $\sKer \alpha$ has support on a proper analytic set in $U$, \ie is a torsion sheaf or the zero sheaf. Since $\cO_U$ does not contain any torsion sheaf, $\alpha$ is a monomorphism. (Alternatively, one can apply \autoref{lem:Monomorphism}.)
\spar
{\spimply 21:}\quad 
By the uniqueness of the minimal resolution (see \eg \Thm 20.2 in \cite{Eisenbud95}), we can assume that $N$ is equal to the minimal number of generators of $\sS$ in $p$, \ie
	\[N = \rk_p \sS \overset{\hbox{\tiny def}}= \cork_p\sS + \rk\sS.\] The injectivity of $\alpha$ implies $N-m = \rk (\cO^N/\alpha(\cO^m)) = \rk \sS$, \ie
\thereqed{1.5\baselineskip}
	\[m=N-\rk\sS=\cork_p\sS.\]
\end{proof}

\spar

\begin{rem}\label{rem:CohenMacaulay} We will work with Cohen-Macaulay spaces (for a definition and some crucial properties, see \eg \S\,5 in \cite{PeternellRemmert94}). Let us recall the following facts about Cohen-Macaulay spaces which will be used below:
\begin{itemize}
\item[(i)] A complex space $X$ is Cohen-Macaulay in $p\in X$ if and only if for any (or at least one) non-zero-divisor $f$ in the maximal ideal sheaf $\am_p$, $\{f=0\}$ is Cohen-Macaulay in $p$.
\item[(ii)] If $X$ is Cohen-Macaulay and $A$ is an analytic subset of $X$ with $\codim A \gr 2$, then $\cO(X)\abb \cO(X\minus A)$ is bijective.
\item[(iii)] A Cohen-Macaulay space is normal if and only if its singular set is at least 2-codimensional.
\end{itemize}\end{rem}

\begin{lem}\label{lem:Hypersurfaces}
Let $X$ be a normal or Cohen-Macaulay space and $S\aus X\stimes\CC^N$ be a linear space over $X$ with at least 2-codimensional singular locus in $X$ and defined by one fiberwise linear function $h \in\cO(X\stimes\CC^N)$. Then $S$ is locally irreducible. In particular, the coherent analytic sheaf associated to $S$ is torsion-free.\end{lem}

\begin{proof} 
Let $A\aus X$ denote the singular locus of $S$ (as linear space) and $E$ denote the primary component of $S$.
\autoref{lem:Homogeneous} implies that $E$ is given by the ideal sheaf $(h, g_1,...,g_m)$ with $g_i$ holomorphic on $X\stimes\CC^N$ and fiberwise  homogeneous (shrink $X$ if necessary). On the regular part $X' := X\minus A$ of $S$, we get $S_{X'}=E_{X'}$, \ie $g_{i,(p,z)} \in (h)_{(p,z)}\ \forall (p,z)\in X'\stimes\CC^N$. Therefore, $f_i:=g_i/h_i$ is a holomorphic function on $X'\stimes \CC^N$. Since we assumed $X$ to be normal or Cohen-Macaulay and $A$ is of codimension 2 in $X$, $f_i$ can be extended to a holomorphic function on $X\stimes \CC^N$. 
We obtain $g_i \in (h)$ and $E=S$.
Now, \autoref{lem:TorsionImpliesRed} implies the second statement.
\end{proof}

Note that for the proof of \autoref{lem:Hypersurfaces}, we hardly used the fact that $S$ is given by a principal ideal sheaf.
If $S$ is defined by more than two functions while the corank of $S$ is 1,
it can happen that $\sS{=}\sL(S)$ has torsion elements with support on a 2-codimensional set. Since the singular locus of a  torsion-free coherent analytic sheaf on a normal complex space is at least 2-codimensional (see \Cor of \cite[\Lem 1.1.8]{ObonekSchneiderSpindler}), we get the following corollary from \autoref{lem:Hypersurfaces}.

\begin{cor} Let $\sS$ be a torsion-free coherent analytic sheaf on a normal complex space  $X$ such that $\cO_X\abb \cO_X^N \abb \sS \abb 0$ is exact (\ie the homological dimension of $\sS$ is at most 1, see \autoref{lem:corank-criterion-homological-dimension}). Then the linear space associated to $\sS$ is locally irreducible.\end{cor}

For a sheaf $\sS$ of homological dimension one, but with arbitrary corank,
the associated linear space $L(\sS)$ is not necessarily locally irreducible (see \autoref{rem:CEIrred}, where the ideal sheaf is of homological dimension one and of corank two).
If $\codim \{x\colon \cork_x\sS\gr k\} > k$ for all $k\,{\gr}\,1$, then $L(\sS)$ is irreducible (see \Prop 3.11 in \cite{Axelsson-Magnusson98}). But, it is not clear whether $L(\sS)$ is reduced. At least, we can prove that $\sS$ is torsion-free:

\begin{lem}\label{lem:homological-torsion} Let $X$ be a normal or Cohen-Macaulay space, and let $\sS$ be a coherent analytic sheaf on $X$ such that the homological dimension of $\sS$ is at most one and  $\Sing \sS$ has at least codimension 2 in $X$. Then $\sS$ is torsion-free.\end{lem}
\begin{proof}
Let $A$ denote the singular locus of $\sS$. By \autoref{lem:corank-criterion-homological-dimension}, every point $p\in X$ has a neighborhood $U$ such that
\begin{eqnarray}\label{eq:new-alpha}
	0\abb \cO_U^{m}\overset{\alpha}\longrightarrow \cO_U^{N} \abb \sS_U\abb 0,
\end{eqnarray}
with $m:=\cork_p\sS$ and $N:=\rk \sS+m$.
Therefore, the associated linear space $S:=L(\sS_U)$ is defined by $m$ holomorphic functions $h_j$ on $U\stimes\CC^N$ which are fiberwise linear.
Let $s$ be a torsion section of $\sS$, \ie there exists an $r \in\cO^\ast_X(U)$ with $r s= 0$ for small enough $U$.
We can understand $s(x,z)=\sum_{i=1}^{N} s_i(x) z_i$ as fiberwise linear holomorphic function on $U_x\stimes\CC_z^{N}$,
and will show that $s$ is zero on $S$, \ie $s$ is in the ideal sheaf generated by $h_j$, $j=1,...,m$:

\spar
$rs=0$ on $S$ implies that there exist $m$ holomorphic functions $a_j$ on $U$ (shrinking $U$ if necessary) such that	
	\[rs=\sum\!{\big.}_{j=0}^m\, a_j h_j.\]
Since $s$ is represented by the tuple $(s_i)$ in $\cO_U^{N}$, $(a_j)\in\cO^m_U$ is the preimage $\alpha^{-1}(r\!\cdot\!(s_i))=r\alpha^{-1}((s_i))$
under $\alpha$ in \eqref{eq:new-alpha}. $(a_j)$ is uniquely determinate since $\alpha$ is injective. The support of $\sT(\sS)$ is contained in $A$, \ie $s=0$ on $S$ over $U':=U\minus A$. In particular, for all $x\in U'$, there exists the decomposition
	\[s_x=\sum\!{\big.}_{j=0}^m\, b_{j,x} h_{j,x}.\]
A priori, $b_j$ depends on $x$. Yet, since they are uniquely determined ($\alpha$ is injective), they are independent and exist on $U'$. In particular, $b_j= \frac{a_j}{r}$ on $U'$. Since $A$ is at least $2$ codimensional and $U\aus X$ is normal or Cohen-Macaulay, $\frac{a_j}{r}$ are holomorphic maps on $U$ (see \eg \autoref{rem:CohenMacaulay} (ii)). Hence,
	\[s =\sum\!{\big.}_{j=0}^m\, \hbox{$\frac{a_j}{r}$} h_j,\]
as desired.
\end{proof}

\bpar
We call a normal complex space \newterm{factorial} if its structure sheaf is factorial (also called unique factorization domain). In this case,
hypersurfaces are (locally) given as the zero set of one holomorphic function. The simplest examples for factorial spaces are manifolds.

\begin{thm}\label{thm:IrredPrincipalIdeal}
Let $S$ be a linear space over a factorial complex space $X$ which is locally defined by one holomorphic fiberwise linear function in $X\stimes \CC^{\rank S +1}$ (\ie the associated ideal sheaf is a principal ideal sheaf). Then the primary component of $S$ is a linear space.\end{thm}

\begin{proof}
Let $S\aus X\stimes \CC^N$  be given by the fiberwise linear $h\in\cO(X\stimes \CC^N)$.
The primary component $E$ of $S$ is an irreducible hypersurface. Since $X$ (and, hence, $X\stimes \CC^N$) is factorial, the ideal sheaf $\sJ_E$ is generated by one element $g$. By \autoref{lem:Homogeneous}, we get $g$ is fiberwise homogeneous. Moreover, $g$ divides $h$. Hence, it has to be fiberwise linear. This implies that $E$ is a linear space.
\end{proof}

\begin{lem}\label{lem:IntersectionOfIrreducibleHypersurfaces}
Let $X$ be a factorial complex space and $S\aus X\stimes\CC^N$ a linear space associated to a torsion-free coherent analytic sheaf on $X$.
Then $S$ can be defined by locally irreducible holomorphic fiberwise linear functions.\end{lem}
\begin{proof} Let the linear space $S$ be defined by fiberwise linear $h_1,.., h_m {\in}\cO(X\stimes\CC^N)$. Let $S_i:=\PC(\{h_i=0\})$ be defined by the fiberwise linear $g_i\in\cO(X\stimes\CC^N)$ (using \autoref{thm:IrredPrincipalIdeal}).  We will prove $S=\bigcap S_i$, \ie $(h_i)_{i=1}^m=(g_j)_{j=1}^m$:\\
Since $g_j|h_j$, we get $(h_i)_{i=1}^m\subset(g_j)_{j=1}^m$.  On the other hand, $g_j$ vanishes on $S_j$. Hence, it vanishes on $\PC(S)$, as well. Since the coherent analytic sheaf $\sL(S)$ associated to $S$ is torsion-free, we get $g_j=0$ on $S$ (in the non-reduced sense; see \autoref{rem:CharTFWithPC}), \ie $g_j\in (h_i)_{i=1}^m$.\end{proof}

\spar
\begin{thm}\label{thm:IrredCorank1} Let $\sS$ be a torsion-free coherent analytic sheaf of corank 1 on a factorial complex space $X$. Then the linear space associated to $\sS$ is locally irreducible and, for small enough open $U\aus X$, there exists an exact sequence
	\[0\abb \cO_U \abb \cO_U^{\rank \sS+1}\abb\sS_U\abb 0,\]
\ie the homological dimension of $\sS$ is at most 1 (see \autoref{lem:corank-criterion-homological-dimension}).\end{thm}

\begin{proof}
\autoref{lem:IntersectionOfIrreducibleHypersurfaces} implies, for small enough open $U\aus X$, that the linear space $S$ associated to $\sS_U$ can be defined by irreducible fiberwise linear $h_1,...,h_m \in \cO(U\stimes\CC^N)$ with $N=\rank \sS+1$. Yet, the primary component $E$ of $S$ is already an irreducible hypersurface in $U\stimes\CC^N$. Hence, $E$ coincides with $S_i:=\{h_i=0\}$ and is a linear space. \autoref{lem:IrredIfPCIsLS} implies $S=E=S_i$.
\spar
We obtain the exact sequence $\cO_U \overset{h_i^\ast}{\longrightarrow} \cO_U^{N}\abb\sS_U\abb 0$. 
\autoref{lem:corank-criterion-homological-dimension}, \spimply12 or \autoref{lem:Monomorphism} give the injectivity of $h_i^\ast$.
\end{proof}

\mpar
Let us generalize this for sheaves with corank 2:

\begin{thm}\label{thm:IrredCorank2}
Let $X$ be a factorial Cohen-Macaulay space and let $S$ be a linear space of corank 2 on  $X$ such that $\Sing S$ has at least codimension 3 in $X$ and the coherent analytic sheaf $\sL(S)$ associated to $S$ is torsion-free. Then $S$ is locally irreducible.
\end{thm}

\begin{proof} 
The proof is similar to the proof of \autoref{lem:Hypersurfaces}.
Let $S\aus U\stimes \CC^N$ be defined by the fiberwise linear $h_1,..., h_m\in\cO(U\stimes\CC^N)$ for an open subset $U\aus X$ with $N=2{+}\rank\sS$.
Because of \autoref{lem:IntersectionOfIrreducibleHypersurfaces}, we can assume that $S_i:=\{h_i=0\}$ is locally irreducible. In particular, $S_i$ is Cohen-Macaulay. Let us assume $h_1, h_2 \ungl 0$ and $h_2\notin (h_1)$. Since $h_1$ is irreducible, $S_{12}:=S_1\cap S_2$ is a linear space with the same rank as $S$. We will prove that $S_{12}$ coincides with $E:=\PC(S)$:
\spar
By \autoref{lem:Homogeneous}, $E$ is defined by $g_1,...,g_k$ (in particular, $h_i \in (g_1,...,g_k)$). Let $A$ denote the singular locus of $S$ and $U' = U\minus A$. Since $E\aus S_{12}$ and $\dim E_{U'} = \dim S_{12, U'}$, we get $E= S_{12}$ over $U'$. Hence, we obtain $g_{j,(p,z)}\in (h_1,h_2)_{(p,z)}$ for all $(p,z)\in S_{1,U'}$. This means $\frac{g_j}{h_2}$ is a holomorphic function on $S_{1,U'}$. By assumption, $\codim_{S_1} ((A{\stimes} \CC^N)\cap S_1) \gr \codim_{U\stimes \CC^N} (A\stimes \CC^N)-1 \gr 2$. Taking into account that $S_1$ is Cohen-Macaulay,
we conclude that $\frac{g_j}{h_2}$ is holomorphic on $S_{1}$ and, hence, $g_j \in (h_1,h_2)$, \ie $E=S_{12}$.
\autoref{lem:IrredIfPCIsLS} implies the claim.
\end{proof}

\begin{cor}\label{cor:IrredCorank2}
Let $X$ be a factorial Cohen-Macaulay space and $\sS$ be a torsion-free coherent analytic sheaf $\sS$ of corank 2 on $X$ with at least 3\nbd codi\-men\-sion\-al singular locus. Then the linear space associated to $\sS$ is locally irreducible and, for small enough open $U\aus X$, there exists an exact sequence
	\[0\abb \cO_U^2 \abb \cO_U^{\rank \sS{+}2}\abb\sS_U\abb 0,\]
\ie the homological dimension of $\sS$ is at most 1.\end{cor}

Keeping the counter-example \autoref{rem:CEIrred} in mind,
one sees the assumption on the codimension is necessary.

\medskip
\begin{proofX}{of \autoref{thm:MainIrred}}
\autoref{lem:TorsionImpliesRed} gives the implication \spimply 21, and \autoref{thm:IrredCorank1} and \autoref{cor:IrredCorank2} yield the implication \spimply 1{3, 4}. \spimply 41 is obtained by \autoref{lem:homological-torsion}.
There is one implications left.

\spar
\spimply 32:\quad Assume that (3) is satisfied, \ie that $S:=L(\sS)$ is locally irreducible.
By definition, this implies that $S$ is reduced. But $S$ is connected. So, there can be just one irreducible component, \ie (2) holds also.
\end{proofX}

\mpar
If the codimension of the singular set of $\sS$ is big enough,
we can prove normality of the linear space associated to $\sS$.

\begin{thm}\label{thm:NormalLS}
Let $X$ be a factorial Cohen-Macaulay space and $\sS$ be a torsion-free coherent analytic sheaf on $X$ of corank at most 2 with $\codim\Sing \sS \gr 2\,{+}\corank \sS$. Then the linear space associated to $\sS$ is normal.\end{thm}

\begin{proof}
Let $S \aus  U\stimes\CC^N$ be the linear space associated to $\sS$ on an open subset $U$ of $X$ with $N=\rank\sS+\corank\sS$. \autoref{thm:IrredCorank1} or \autoref{cor:IrredCorank2} implies that $S$ is Cohen-Macaulay.
Let $A:=\Sing S=\Sing\sS_U\aus U$ denote the singular locus of $S$ as linear space. The singular set $S_\sing$ of $S$ as analytic subset of $U\stimes\CC^N$ is contained in $((A\stimes\CC^N)\cap S) \cup S_{U_{\sing}\minus A}$. We get
	\[\codim_S ((A\stimes\CC^N)\cap S) \gr \codim_U A - \corank \sS \gr 2\]
and $\codim_S S_{U_{\sing}\minus A} \gr \codim U_\sing \gr 2$ since $U\aus X$ is normal and $S_{U\minus A}$ is a vector bundle. Hence, $\codim_S\,S_\sing\gr 2$.
\autoref{rem:CohenMacaulay} (iii) (see \eg \cite[\Cor 5.2]{PeternellRemmert94}) implies that $S$ is actually normal.
\end{proof}

\mpar
\subsection{More preliminaries on torsion} $ $

\mpar
Throughout the paper, we will use the following observation without mentioning explicitly.
Let $\psi\colon \sF \rightarrow \sG$ be a morphism of analytic sheaves on a (locally irreducible) complex space $(X,\cO_X)$.
Then $\psi$ induces a canonical map 
\[\widehat\psi\colon \sF/\sT(\sF) \rightarrow \sG/\sT(\sG)\]
because the torsion sheaf $\sT(\sF)$ of $\sF$ is mapped by $\psi$ into the 
torsion sheaf $\sT(\sG)$ of $\sG$:
$r_x \psi(s_x) = \psi(r_x s_x)=0$ for germs $r_x\in\cO_{X,x}$, $s_x\in\sT(\sF)_x$ with $r_xs_x=0$.
Note that particularly $\sT(\sF) \subset \ker \psi$ if $\sG$ is torsion-free.

\medskip

\begin{lem}\label{lem:Monomorphism} 
Let $X$ be a locally irreducible complex space, $\sF$ and $\sG$ coherent analytic sheaves on $X$
such that there exists a morphism $\psi\colon \sF\abb\sG$ which is a monomorphism on an open dense subset of $X$. If $\sF$ is torsion-free, then $\psi$ is a monomorphism.
If not, $\psi$ induces a monomorphism $\hat\psi\colon\sF/\sT(\sF){\hookrightarrow} \sG/\sT(\sG)$.
\end{lem}

\begin{proof} The second statement follows from the considerations above and the torsion-free case. 
Hence, we can assume that $\sF$ is torsion-free.\\
Let $F$ and $G$ denote the linear spaces associated to $\sF$ and $\sG$, respectively.
\autoref{thm:LocallyIrredOfPC} implies that $\PC(F)$ and $\PC(G)$ are locally irreducible. Let $\psi$ be a monomorphism on the open dense subset $W$ of $X$ with $W\aus X\minus (\Sing \sF \cap\Sing\sG)$.
Thus, $\psi$ induces a holomorphic fiberwise linear map $\psi^\ast\colon G\abb F$ such that $\psi^\ast_W\colon G_W\abb F_W$ is a surjective map of vector bundles.
Let $s$ be a section in $\sKer\psi$, \ie $\psi^\ast \nach s$ vanishes on $G$. We get that $s$ vanishes on $F_W$ and, hence, on $\PC(F)$. Since $\sF$ is torsion free, we obtain $s=0$ (\cf \autoref{rem:CharTFWithPC}).
\end{proof}

Alternatively, one can prove this lemma by using only sheaf-theoretical terminology and arguments (\cf the proof of \autoref{lem:Injection2}).

\medskip
Note the following trivial consequence of the definition of the pullback:

\begin{lem}\label{lem:einbettung} 
Let $S\aus U\stimes \CC^N$ be a linear space over a complex space $U$, and let $\pi\colon V\abb U$ be a holomorphic map. 
Then the pullback $\pi^\ast(S)= V{\times_U} S$ can be embedded in $V\stimes\CC^N$.
\end{lem}

\begin{rem}\label{rem:Rabinowitz}
Let $\sS$ be a coherent analytic sheaf over a locally irreducible complex space $X$. 
Let $\ph=\ph_\sS\colon X_\sS\rightarrow X$ be the monoidal transformation of $X$ with respect to $\sS$, \ie $\sE:=\ph^T \sS$ is locally free.
Note that $X_\sS$ is again locally irreducible.
Then $\ph^\ast\sS$ has torsion in a point $q$ if and only if $\sS$ is not locally free in $\ph(q)$ (see \cite{Rabinowitz79}).
We will give a short, alternative proof with the statements from above.
Let $S$, $S^\ast$ and $E$ denote the linear complex spaces associated to $\sS$, $\ph^\ast\sS$ and $\sE$, respectively. 
If $\sS$ is not locally free in $\ph(q)$, then $\dim E_q<\dim S_{\ph(q)}$
(as $\dim E_q=\rk E=\dim S_{\tilde q}$ in all points $\tilde q$ where $\sS$ is locally free, see \S\,1.1 in \cite{GrauertRiemenschneider70}). 
\autoref{lem:einbettung} implies $\dim S^\ast_q=\dim S_{\ph(q)}>\dim E_q$. 
Since $\PC(S^\ast)=E$, we obtain that $S^\ast$ is reducible in $(q,0)$, \ie $\ph^\ast\sS$ has torsion in $q$ by \autoref{lem:IrredIfPCIsLS}.
The other implication of the claim is trivial.
\end{rem}

\section{Direct images of torsion-free preimage sheaves}
\label{sec:Injection1}

In this section, we will prove \autoref{thm:MainNormal} and more:

\begin{thm}\label{thm:Injection1+MainNormal}
Let $X$ be a locally irreducible complex space, $\sS$ a torsion-free coherent analytic sheaf on $X$
and $\pi\colon Y\abb X$ a proper modification of $X$. 
\begin{itemize} \item[(i)]
Then the canonical homomorphisms $\sS\rightarrow \pi_*(\pi^* \sS)$ and $\sS \rightarrow \pi_*(\pi^T \sS)$ both are injective, 
where $\pi^T \sS$ is the torsion-free preimage sheaf of $\sS$ under $\pi$.
\item[(ii)]
If the linear space $L(\sS)$ associated to $\sS$ is locally irreducible and $L(\pi^\ast\sS)$ is reduced, then $\pi_*(\pi^*\sS) \rightarrow \pi_*(\pi^T \sS)$ is injective.
\item[(iii)]
If the linear space $L(\sS)$ is normal, then
	\[\sS \iso \pi_\ast (\pi^T \sS).\]
\end{itemize}\end{thm}

\begin{proof}
We can assume that $X$ is connected.
Let $S$ denote the linear space associated to $\sS$, $A\aus Y$ the set where $\pi$ is not biholomorphic and $A^c$ the complement. 
$S^\ast=Y {\times_X} S$ is the linear space associated to $\sS^\ast:=\pi^\ast\sS$. Let $\pr\colon S^\ast\abb S$ denote the projection,
let $E$ be the linear space associated to $\sE:=\pi^T \sS$, let $U$ be an open Stein set in $X$, and let $V:=\pi^{-1}(U)$. The construction of the linear spaces implies
	\[\begin{split}\Hom(S_U, U\stimes \CC)&\iso\sS(U),\\
	\Hom(S^\ast_V, V\stimes \CC)&\iso\pi^*\sS(V)=\big(\pi_\ast(\pi^* \sS)\big)(U) \FErgT{and}\\
	\Hom(E_V, V\stimes \CC)&\iso\sE(V)=(\pi_\ast\sE)(U).\end{split}\]
\mpar
Let $N$ be an integer big enough, so that $S_U$ can be realized as a subset of $U\stimes \CC^N$. 
We obtain closed embeddings $E_V\aus S_V^\ast\aus V\stimes\CC^N$, 
and $(q,z)$ is in $S^\ast_V$ if and only if $(\pi(q),z)\in S$. Since obviously $\pr(E_{A^c})=S_{\pi(A^c)}$ and proper holomorphic images of irreducible sets are irreducible, we obtain (\cf \eqref{eq:praeSurj})
\begin{equation}\label{eq:Surj}
\pr(\PC(E_V))=\PC(S_U).
\end{equation}

\medskip
\looseitem{(i)} $\sS\hookrightarrow \pi_*(\pi^* \sS)$ and $\sS \hookrightarrow \pi_*(\pi^T \sS)$: This follows from \autoref{lem:Monomorphism}.

\medskip
\looseitem{(ii)} Assume $S$ is locally irreducible and $S^\ast$ is reduced. To prove that the natural map given by the restriction $\Hom(S^\ast_V, V\stimes \CC)$ $\abb$ $\Hom(E_V,$ $V\stimes \CC)$ is injective, we apply $\pr(\PC(E_V))=S_U$ 
(use \eqref{eq:Surj} and that $S$ is irreducible):
\spar
Let $s\in \Hom(S^\ast_V, V\stimes \CC)$ with $s|_E=0$, \ie $s|_{\PC (E)}=0$. 
Let $s(q,z)=(q,f(q,z))\in\Hom(S^\ast_V, V\stimes \CC)$ be not the zero section,
\ie there is a point $(q',z_0)\in S^\ast$ with $f(q',z_0)\ungl 0$. There is a $q''\in\pi^{-1}(\pi(q'))$ such that $(q'',z_0)\in E$. Since $\pi^{-1}(\pi(q'))\stimes \{z_0\}$ is a compact analytic set in $S^\ast_V$, we get $f(q',z_0)=f(q'',z_0)$, \ie $f|_E \ungl 0$ and $s|_E \ungl 0$.

\medskip
\looseitem{(iii)} Assume that $S$ is normal. 
Fix a section $s(q,z){=}(q,f(q,z)) \in \Hom(E_V,$ $V\stimes \CC)$. 
Since $\pr\colon \PC(E_V)\abb \PC(S_U)=S_U$ is a proper modification (surjectivity is \eqref{eq:Surj}),
the map $\tilde f:=f \circ \pr^{-1}\colon S_U \abb \CC$ is a bounded meromorphic function, \ie it is weakly holomorphic. 
Since $S_U$ is normal, $\tilde f$ is holomorphic. 
Obviously, it is linear in the second argument. 
Hence, $\pr^{-1}$ gives a map $(\pr^{-1})^*\colon \Hom(E_V, V\stimes \CC) \abb \Hom(S_U, U\stimes \CC), s\auf \tilde s$ with $\tilde s(p,z)=(p,\tilde f(p,z))$. Since $f\circ \pr^{-1}=\tilde f=0$ implies $f=0$, this map is injective.
It is now easy to see that $(\pr^{-1})^*\colon \pi_*\pi^T \sS \hookrightarrow \sS$ is the inverse to the natural mapping $\sS \hookrightarrow \pi_*\pi^T \sS$.
\end{proof}

\begin{rem}\label{rem:CE1}
Without the additional assumption about normality, 
the natural map $\sS \hookrightarrow \pi_*\pi^T \sS$ is not necessarily bijective.
The following counter-example is derived from one due to Mircea Musta\c{t}\u{a}.
\spar
Let $\sS=(x^3,y^3)$ be the ideal sheaf on $\CC^2_{x,y}$ generated by the functions $x^3$ and $y^3$, 
and let $\pi\colon M\abb \CC^2$ be the blow up of the origin, \ie 
\[M=\{ (x,y; [t_1:t_2])\in \CC^2\stimes\CP^1\colon xt_2=yt_1\}.\]
Then
\[S=L(\sS) = \{(x,y;z_1,z_2) \in \CC^2\stimes\CC^2\colon z_2x^3=z_1y^3\},\]
and sections in $\sS$ correspond to sections in $\Hom(S,\CC^2\stimes\CC)$
via the assignment $x^3 \mapsto [(x,y;z_1,z_2) \mapsto z_1]$, $y^3\mapsto [(x,y;z_1,z_2) \mapsto z_2]$.
Now,
\begin{equation*}\begin{split}S^\ast&=\makebox[3.4em]{$L(\pi^* \sS)$} = \{(x,y;[t_1:t_2];z_1,z_2) \in \CC^2\stimes\CP^1\stimes\CC^2\colon xt_2=yt_1,z_2x^3=z_1y^3\},\\
E&=\makebox[3.4em]{$L(\pi^T \sS)$} = \{(x,y;[t_1:t_2];z_1,z_2) \in \CC^2\stimes\CP^1\stimes\CC^2\colon xt_2=yt_1,z_2t_1^3=z_1t_2^3\}.\end{split}\end{equation*}
Thus, $S^*=E\cup T$ with
\[T= \{(x,y;[t_2:t_2];z_1,z_2) \in \CC^2\stimes\CP^1\stimes\CC^2\colon x=y=0\}.\]
In $\Hom(E,M\stimes\CC)$, we have now also the section 
\[\big\{\!({\scriptscriptstyle\frac{t_2}{t_1}}z_1\colon t_1\neq0); ({\scriptstyle\frac{t_1^2}{t_2^2}} z_2\colon t_2\neq0)\big\},\]
corresponding to $x^2y$ in $\pi^T\sS$, and the section 
\[\big\{\!({\scriptscriptstyle \frac{t_2^2}{t_1^2} }z_1\colon t_1\neq0); ({\scriptscriptstyle\frac{t_1}{t_2}} z_2\colon t_2\neq0)\big\},\]
corresponding to $xy^2$, but these two do not extend to $S^*=E\cup T$ because there is no relation between $z_1$ and $z_2$ on $T$.
In $\Hom(S^*,M\stimes\CC)$, however, we have the section 
\[\big\{\!(y{\scriptscriptstyle\frac{t_2}{t_1}}z_1\colon t_1\neq0); (y{\scriptscriptstyle\frac{t_1^2}{t_2^2}} z_2\colon t_2\neq0)\big\},\]
corresponding to $x^2y^2$ in $\pi^*\sS$.
Moreover, it is easy to check that $x$, $x^2$, $y$, $y^2$, $xy$ are neither contained in $\pi_*\pi^*\sS$ nor in $\pi_*\pi^T\sS$.
Hence, we have:
\[\sS \subsetneq (x^3,x^2y^2, y^3)=\pi_\ast\pi^\ast\sS\subsetneq (x^3,x^2y,xy^2, y^3)=\pi_\ast\pi^T\sS.\]
\spar
Let us present a counter-example for (ii) in \autoref{thm:Injection1+MainNormal} if $S$ is not irreducible:\\
Let $X=\{x^3=y^2\}\aus\CC^2$ be the cusp, $\pi\colon\CC \abb X$, $\pi(t):=(t^2,t^3)$ the normalization
and let $\hat\cO_X$ denote the sheaf of weakly holomorphic functions on $X$.
Then one can compute that $\pi_\ast\pi^\ast\hat\cO_X$ has torsion elements with support in 0.
Yet, $\pi_\ast\pi^T\hat\cO_X = \pi_\ast\cO_\CC =\hat\cO_X$ is torsion-free.
Hence, there can not exist an injective morphism $\pi_\ast\pi^\ast\hat\cO_X\abb\pi_\ast\pi^T\hat\cO_X$.
\end{rem}

\begin{rem} 
In order to get the isomorphism $\sS = \pi_* \pi^T \sS$, normality is a natural assumption.
For example, if $\sE$ is locally free, we obtain $\pi_\ast\pi^\ast \sE \iso \sE$ if and only if $\pi_\ast \cO_Y\iso \cO_X$. 
\end{rem}

\section{Torsion-free preimages of direct image sheaves}
\label{sec:Injection2}

\begin{lem}\label{lem:Injection2}
Let $\pi\colon Y\rightarrow X$ be a proper modification between complex spaces $Y$, $X$,
and $\sE$ a torsion-free coherent analytic sheaf on $Y$. 
Then the canonical homomorphism $\pi^*\pi_* \sE \rightarrow \sE$ induces a canonical injection
\begin{equation}\label{eq:Injection2}
\pi^T\pi_* \sE \hookrightarrow \sE,
\end{equation}
where $\pi^T\pi_*\sE = \pi^*\pi_*\sE / \sT(\pi^*\pi_* \sE)$ is the torsion-free preimage of $\pi_* \sE$.
\end{lem}

The following proof was communicated to us by Matei Toma. Alternatively,
\autoref{lem:Injection2} follows also from \autoref{lem:Monomorphism}.

\smallskip
\begin{proof} 
Let $\sT$ denote the torsion sheaf of $\pi^\ast\pi_\ast\sE$, and let $\psi\colon\pi^\ast\pi_\ast\sE\abb \sE$  denote the natural map. 
Since $\sE$ is torsion-free, $\psi(\sT)=0$ and, hence, $\psi$ factors through $\hat\psi\colon\pi^T\pi_\ast \sE \abb \sE$. 
Since $\psi$ is an isomorphism outside of a thin analytic set $A\aus X$, an element in the kernel of $\psi$ has support in $A$. 
Therefore, the kernel is a subset of $\sT$, \ie $\hat \psi$ is injective.
\end{proof}

\begin{rem}\label{rem:CE2}
Let us give a counter-example showing that \eqref{eq:Injection2} is in general not an isomorphism.
Consider a modification $\pi\colon M \rightarrow \CC^n$ where $M$ is a complex manifold
with canonical sheaves $\Omega^n_M$ and $\Omega^n_{\CC^n}\iso\cO_{\CC^n}$. 
Then $\pi_* \Omega^n_M=\Omega^n_{\CC^n}\iso\cO_{\CC^n}$ so that $\pi^T \pi_* \Omega^n_M \iso \pi^T \cO_{\CC^n} \iso \cO_{M}$.
But $\cO_M {\neq} \Omega^n_M$ in general.
\end{rem}

However, we can be a bit more precise in \autoref{lem:Injection2} by use of the following observation 
if $\sE$ is locally free of rank 1:

\begin{lem}
Let $X$ be a complex space and $i\colon \sF \hookrightarrow \sG$ an injective morphism 
between two coherent locally free sheaves of rank 1 over $X$.
Then there exists an effective Cartier divisor, $D\geq 0$, such that
\[i(\sF) = \sG \otimes \cO_X(-D).\]
In particular, $i$ is an isomorphism precisely on $X-|D|$.
\end{lem}

\begin{proof}
Let $\{X_\alpha\}_\alpha$ be a locally finite open cover of $X$ such that both, $\sF$ and $\sG$, are free over each $X_\alpha$.
So, there are trivializations
\begin{eqnarray*}
\phi_\alpha\colon \sF|_{X_\alpha} &\overset{\sim}{\longrightarrow}& \cO_{X_\alpha},\\
\psi_\alpha\colon \sG|_{X_\alpha} &\overset{\sim}{\longrightarrow}& \cO_{X_\alpha},
\end{eqnarray*}
and for $X_{\alpha\beta}:=X_\alpha\cap X_\beta\neq \emptyset$, 
we have transition functions $F_{\beta\alpha}:= \phi_\beta\circ \phi_\alpha^{-1} \in \cO^*(X_{\alpha\beta})$
and $G_{\beta\alpha}:= \psi_\beta\circ \psi_\alpha^{-1} \in \cO^*(X_{\alpha\beta})$ satisfying the cocycle conditions.
In trivializations 
\begin{eqnarray*}
\psi_\alpha \circ i|_{X_\alpha} \circ \phi_\alpha^{-1}\colon \cO_{X_\alpha} \rightarrow \cO_{X_\alpha}
\end{eqnarray*}
is given by a holomorphic function $i_\alpha\in \cO(X_\alpha)$, vanishing nowhere identically,
with (unreduced) divisor $(i_\alpha)$. It is easy to see that
$G_{\beta\alpha} \cdot i_\alpha = i_\beta \cdot F_{\beta\alpha}$
on $X_{\alpha\beta}$, so that $i_\alpha / i_\beta = F_{\beta\alpha} / G_{\beta\alpha} \in \cO^*(X_{\alpha\beta})$.
Thus $D:=\{(X_\alpha,i_\alpha)\}_\alpha$ defines in fact an effective Cartier divisor with support $|D|$.

To see that $i(\sF)=\sG\otimes\cO_X(-D)$, note that
$\sG\otimes \cO_X(-D)$ is a coherent subsheaf of $\sG$ because $\cO_X(-D)$ is a sheaf of ideals in $\cO_X$,
and that
\thereqed{1.8\baselineskip}
\[\psi_\alpha\otimes 1\colon \left.\sG\otimes \cO_X(-D)\right|_{X_\alpha} \overset{\sim}{\longrightarrow} \cO_{X_\alpha} \otimes \cO_{X_\alpha}(-(i_\alpha)).\]
\end{proof}

So, we can deduce the following direct consequence of \autoref{lem:Injection2}:

\begin{thm}\label{thm:key3}
Let $\pi\colon Y\rightarrow X$ be a proper modification of $X$,
$\sE$ a locally free analytic sheaf of rank 1 on $Y$ and assume that $\pi^T\pi_* \sE$ is also locally free.
Then there exists an effective Cartier divisor $D$ on $Y$ such the following holds:
The canonical homomorphism $\pi^*\pi_* \sE \rightarrow \sE$ induces a canonical injection
\[i\colon \pi^T\pi_* \sE \hookrightarrow \sE\]
and
\[i(\pi^T\pi_* \sE) = \sE\otimes \cO_Y(-D).\]
In particular, $i$ is an isomorphism precisely on $Y-|D|$, and $|D|$ is contained in the exceptional set of $\pi$.
\end{thm}

\section{\texorpdfstring{$\pi^T$}{Pi\textasciicircum T} preserves injectivity and surjectivity}
\label{sec:mainA}

In this section, we consider properties of $\pi^T$ as functor and use them to prove \autoref{thm:mainA}.

\begin{lem}\label{lem:quasiexactness}
Let $\pi\colon Y \rightarrow X$ be a proper modification of a complex space $X$ such that $Y$ is locally irreducible.
Let
\[\psi\colon \sF \rightarrow \sG\]
be a morphism of coherent analytic sheaves. If $\psi$ is an epimorphism,
then the induced mapping 
\[\pi^T\psi\colon \pi^T \sF \rightarrow \pi^T \sG\]
is also an epimorphism. If $\psi$ is a monomorphism, then $\pi^T\psi$ is also a monomorphism.
\end{lem}

\begin{proof}
Let $\psi$ be an epimorphism, \ie surjective.
Recall that $\pi^*$ is right-exact. So, $\pi^*\psi\colon \pi^* \sF \rightarrow \pi^*\sG$ is still surjective. 
But then it is easy to see that the induced mapping $\pi^T\psi\colon \pi^T\sF \rightarrow \pi^T\sG$ is also surjective.

\medskip
For the second statement, let $\psi$ be injective.
Let $f_x\in (\pi^T \sF)_x$ such that $\pi^T \psi(f_x)=0$. This means that there is an open set $U\subset Y$ and a representative
$f\in \pi^T\sF(U)$ such that $\pi^T\psi(f)=0$. But $\pi^T \psi$ is injective on a dense open subset $W\subset X$.
Thus, $f=0$ on $U\cap W$, \ie $f$ has support on a thin set. But $\pi^T \sF$ is torsion-free. So, $f_x=0$ and $f=0$. 
\end{proof}

Note that $\pi^T$ is not exact. A simple counter-example is as follows.
Let $\mathfrak{m}$ be the maximal ideal sheaf of the origin in $\CC^2$.
Then $0\rightarrow \mathfrak{m} \hookrightarrow \cO_{\CC^2} \rightarrow \cO_{\CC^2}/\mathfrak{m} \rightarrow 0$ is exact.
Let $\pi$ be just the identity on $\CC^2$. So,
we have $\pi^T \mathfrak{m}=\mathfrak{m}$, $\pi^T \cO_{\CC^2} = \cO_{\CC^2}$ and
$\pi^T \big( \cO_{\CC^2} / \mathfrak{m}\big) =0$.
The resulting sequence $0 \rightarrow \mathfrak{m} \hookrightarrow \cO_{\CC^2} \rightarrow 0$ is clearly not exact.

\bigskip
\begin{proofX}{of \autoref{thm:mainA}}
Let $\pi\colon Y\abb X$ be a proper modification between locally irreducible complex spaces. 
Let $\sF$ and $\sG$ be coherent analytic sheaves on $X$ and $Y$, respectively.

\medskip
\looseitem{(i)} The case $\sF=\pi_\ast \sG$: By \autoref{thm:Injection1+MainNormal} (i), the natural map $\sF\rightarrow \pi_\ast\pi^T \sF$ is injective. 
Moreover, \autoref{lem:Injection2} yields injectivity of the natural map
	\[\pi^T\sF=\pi^T\pi_\ast \sG\rightarrow \sG.\]
Since $\pi_\ast$ is left-exact, we obtain the second natural injection
	\[\pi_\ast\pi^T\sF\hookrightarrow \pi_\ast\sG=\sF.\]

\medskip
\looseitem{(ii)} The case $\sG=\pi^T \sF$: As above, \autoref{lem:Injection2} gives $\pi^T\pi_*\sG\hookrightarrow \sG$. 
By \autoref{thm:Injection1+MainNormal} (i) we have also the natural injection
	\[\sF\hookrightarrow \pi_\ast\pi^T\sF=\pi_\ast\sG.\]
But $\pi^T$ preserves injectivity (\autoref{lem:quasiexactness}) so that we obtain the injection
\thereqed{1.8\baselineskip}
	\[\sG=\pi^T\sF\hookrightarrow \pi^T \pi_\ast\sG.\]
\end{proofX}

\section{Application to ideal sheaves}
\label{sec:application1}

In this section, we discuss the application of \autoref{thm:mainA} to reduced ideal sheaves.
As a preparation, note the following:

\begin{lem}\label{lem:prepX}
Let $\pi\colon Y \rightarrow X$ be a holomorphic mapping between complex spaces $Y$, $X$.
Then $\pi^* \cO_X = \cO_Y$.
\end{lem}

\begin{proof}
As $\pi^{-1} \cO_X \subset \cO_Y$, we have that $\pi^* \cO_X = \pi^{-1}\cO_X \otimes_{\pi^{-1} \cO_X} \cO_Y = \cO_Y$,
because $\pi^{-1} \cO_X$ contains (the germ of) the function 1 at any point of $Y$.
\end{proof}

Coming to ideal sheaves,
let us start with the following observation:

\begin{lem}\label{lem:idealsheaves1}
Let $X$ be a locally irreducible complex space and $A\subset X$ an (unreduced) analytic subspace with ideal sheaf $\sJ_A$.
Let $\pi\colon Y \rightarrow X$ be a proper modification and $B:=\pi^{-1}(A)$ the unreduced analytic preimage 
with ideal sheaf $\sJ_B$. Then:
\begin{eqnarray*}
\pi^T \sJ_A = \sJ_B.
\end{eqnarray*}
\end{lem}

\begin{proof}
Consider the short exact sequence of sheaves over $X$:
\begin{eqnarray*}
0 \rightarrow \sJ_A \overset{\alpha}{\longrightarrow} \cO_X \longrightarrow \cO_X/\sJ_A\rightarrow 0.
\end{eqnarray*}
By right-exactness of $\pi^*$, we deduce the exact sequence
\begin{eqnarray*}
 \pi^* \sJ_A \overset{\pi^*\alpha}{\longrightarrow} \pi^* \cO_X \longrightarrow \pi^* \big(\cO_X/\sJ_A\big) \rightarrow 0.
\end{eqnarray*}
Now, we use \autoref{lem:prepX} twice: $\pi^* \cO_X=\cO_Y$ and $\pi|_B^* \cO_A = \cO_B $ (which implies that $\pi^* (\cO_X/\sJ_A) = \cO_Y/\sJ_B$
using the definition of the analytic preimage, see \eg \Prop 0.27 in \cite{Fischer76}).
As $\cO_Y$ is torsion-free, it is clear that 
\begin{eqnarray}\label{eq:xxx}
\sT(\pi^* \sJ_A) \subset \sKer\pi^*\alpha.
\end{eqnarray}
Consider $\pi^T \alpha\colon \pi^T \sJ_A \rightarrow \pi^T\cO_X = \pi^* \cO_X=\cO_Y$.
By \eqref{eq:xxx}, it follows that $\sIm(\pi^*\alpha)=\sIm(\pi^T\alpha)$,
and \autoref{lem:quasiexactness} tells us that $\pi^T\alpha$ is injective. So, we obtain the short exact sequence
\begin{eqnarray*}
0\rightarrow \pi^T \sJ_A \overset{\pi^T \alpha}{\longrightarrow} \cO_Y \longrightarrow \cO_Y/\sJ_B \rightarrow 0,
\end{eqnarray*}
telling us that in fact $\pi^T \sJ_A = \sJ_B$.
\end{proof}

It is clear that $\sJ_A$ (and $\pi^T\sJ_A=\sJ_B$) are torsion-free,
and so we obtain from \autoref{thm:mainA} (ii) that:
\begin{eqnarray}\label{eq:appl100}
\sJ_B \cong \pi^T \pi_* \sJ_B.
\end{eqnarray}

\bigskip
Under some additional assumptions, we have also:

\begin{lem}\label{lem:idealsheaves}
Let $X$ be a normal complex space,
and let $A$ be a locally complete intersection or a normal analytic set in $X$ with (reduced) ideal sheaf $\sJ_A$. 
Let $\sigma\colon\tilde X \abb X$ denote the blow up of $X$ with center $A$, i.e. the monoidal transformation with respect to $\sJ_A$,
and let $\sJ_B$ be the (reduced) ideal sheaf associated to the reduced exceptional set $B:=\sigma^{-1}(A)$. Then:
\begin{equation}\label{eq:DI}
\sJ_A \iso \sigma_\ast\sJ_B.
\end{equation}
\end{lem}

\begin{proof}
The statement is local with respect to $X$,
so we can assume that $A$ is the zero-set of reduced holomorphic functions $f_0,...,f_m$. 
$(\sJ_A)_p$ is generated by the germs $f_{0,p},..,f_{m,p}$, and $\tilde X \aus X\stimes \CP^m$
(see \eg \S\,2.5 in \cite{Riemenschneider71} for the monoidal transformation of ideal sheaves).
We show that
	\[\cO_A\iso\sigma_\ast\cO_B.\]

\looseitem{(I)} $A$ is a complete intersection, \ie $m+1=\codim A$: This implies $B = A\stimes \CP^m$. 
For all open subsets $U\aus A$, we obtain
	\[\cO_A(U)\iso \cO_{A\stimes \CP^m}(U\stimes \CP^m)=\cO_B(\sigma^{-1}(U)).\]

\looseitem{(II)} $A$ is normal: In this case, $B$ is an analytic subset of $A\stimes \CP^m$,
and by the surjectivity $\sigma(B)=A$, we get the injection $\cO_A \iso \sigma_\ast \cO_{A\stimes\CP^m} \hookrightarrow \sigma_\ast \cO_B$. 
On the other hand, a section in $\cO_B(\sigma^{-1}(U))$ gives a weakly holomorphic function on $A$: 
With part (I) applied on the regular part $A_\reg$ of $A$, 
we get a holomorphic function on $A_\reg$ which is bounded in points of $A_\sing$. 
Since we assumed $A$ to be normal, we get $\sigma_\ast\cO_B\iso \widehat\cO_A\iso \cO_A$.

\medskip
Thus, $\cO_A\iso \sigma_\ast \cO_B$ as desired. In other words: $\cO_X/\sJ_A\iso \sigma_\ast (\cO_{\tilde X}/\sJ_B)$. 
We obtain the exact commutative diagram:
	\[\begin{xy} \xymatrix@-0.25pc{ 
0\ar[r]& \sigma_\ast \sJ_B\ar[r]& \sigma_\ast \cO_{\tilde X}\ar[r]\ar@{=}^{\wr}[d]& \sigma_\ast(\cO_{\tilde X}/\sJ_B)\ar@{=}^{\wr}[d]&\\
	0\ar[r]& \sJ_A\ar[r]& \cO_{X}\ar[r]& \cO_X/\sJ_A\ar[r]&0.}\end{xy}\]
It follows that in fact $\sigma_\ast \sJ_B\iso \sJ_A$.
\end{proof}
\spar
In the situation of \autoref{lem:idealsheaves}, we can now apply \autoref{thm:mainA} (i) to $\sJ_A\cong\sigma_* \sJ_B$ and obtain:
\begin{eqnarray}\label{eq:appl101}
\sJ_A \iso \sigma_\ast \sigma^T\sJ_A.
\end{eqnarray}

\section{Holomorphic \texorpdfstring{$n$}{n}-forms on singular spaces}
\label{sec:hnforms}

As a consequence of \autoref{thm:mainA} and \autoref{thm:key3}, we get also the following application to holomorphic $n$-forms:

\begin{thm}\label{thm:nforms}
Let $X$ be a complex space of pure dimension $n$ and $\sK_X$ the Grauert-Riemenschneider canonical sheaf on $X$.
Then there exist a resolution of singularities $\pi\colon M\rightarrow X$ 
and an effective divisor, $D\geq 0$, with support on the exceptional set of the resolution
such that
\begin{eqnarray}\label{eq:pKX}
\pi^T \sK_X \cong \Omega_M^n(-D)=\Omega_M^n\otimes \cO_M(-D),
\end{eqnarray}
where $\Omega^n_M$ is the canonical sheaf of holomorphic $n$-forms on $M$, 
and \eqref{eq:pKX} is induced by the natural mapping $\pi^* \sK_X = \pi^*\pi_* \Omega^n_M \rightarrow \Omega^n_M$. 
Moreover, we get
	\[\pi_\ast \Omega_M^n =\sK_X\iso\pi_\ast \Omega_M^n(-D).\]
\end{thm}

\begin{proof}
Let $\pi\colon M \rightarrow X$ be a resolution of singularities such that $\pi^T \sK_X$ is locally free.
Such a resolution exists due to Rossi and Hironaka (apply first Rossi's \autoref{thm:rossi} and then Hironaka's resolution of singularities).
Recall that $\sK_X =\pi_* \Omega^n_M$
by definition of the Grauert-Riemenschneider canonical sheaf.
So, the assertion follows directly from \autoref{thm:key3} and \autoref{thm:mainA}.
\end{proof}

The following observation is also useful:

\begin{lem}\label{lem:tensor}
Let $\sF$, $\sG$ be torsion-free coherent analytic sheaves on a locally irreducible complex space $X$ and let $\pi\colon Y\rightarrow X$ be a proper modification of $X$ such that $\pi^T \sG$ is locally free. Then
\[\pi^T \big( \sF \otimes \sG\big) = \pi^T \sF \otimes \pi^T \sG\]
and there is a natural injection
\[\sF \otimes \sG \hookrightarrow \pi_* \big( \pi^T \sF \otimes \pi^T \sG\big).\]
\end{lem}

\begin{proof}
Note that $Y$ is also locally irreducible.
Consider the two natural surjections $\pi^* \sF \rightarrow \pi^T \sF$ and $\pi^* \sG \rightarrow \pi^T \sG$.
These yield a natural surjection 
\[\pi^* (\sF\otimes \sG) = \pi^*\sF \otimes \pi^* \sG \longrightarrow \pi^T \sF \otimes \pi^T\sG\] 
which is an isomorphism on an open dense subset of $Y$.
Since the tensor product of a torsion-free and a locally free sheaf is torsion-free,\!%
\footnote{The tensor product of two torsion-free sheaves need not be torsion-free. E.\,g., let $\sI$ be the ideal sheaf generated of $(z^2, zw)$ on $\CC^2_{z,w}$ and $\sJ$ be the ideal sheaf generated by $(w^2, zw)$. Then $z^2\stensor w^2- zw\stensor zw \in\sI\stensor\sJ$ is not zero; yet, $z \cdot (z^2\stensor w^2- zw\stensor zw)=z^3\stensor w^2- z^2\stensor zw^2= 0$.}
we obtain by use of \autoref{lem:Monomorphism} a natural isomorphism
\begin{eqnarray}\label{eq:id100}
\pi^T \big( \sF\otimes \sG \big) = \frac{ \pi^* (\sF \otimes \sG)}{\sT\big( \pi^* (\sF \otimes \sG )\big)} 
\overset{\sim}{\longrightarrow} \pi^T \sF \otimes \pi^T\sG.
\end{eqnarray}
The second statement follows by taking the direct image of \eqref{eq:id100} under $\pi$ 
and \autoref{thm:Injection1+MainNormal} (i).
\end{proof}

This lemma and the projection formula gives directly the following corollary of \autoref{thm:nforms}.

\begin{thm}\label{thm:app1}
Let $X$ be a complex space of pure dimension $n$, $\sK_X$ the Grauert-Riemenschneider canonical sheaf on $X$
and $\sF$ a torsion-free coherent analytic sheaf on $X$.
Then there exists a resolution of singularities $\pi\colon M\rightarrow X$ 
and an effective divisor, $D\geq 0$, with support on the exceptional set of the resolution
such that
	\[\pi^T(\sF\tensor\sK_X )\iso \pi^T\sF \otimes \Omega^n_M(-D).\]
If $\sF$ is locally free, then
	\[\pi_\ast(\pi^\ast\sF\tensor\Omega_M^n)\iso\sF\otimes \sK_X \iso \pi_* \left( \pi^\ast \sF \otimes \Omega^n_M(-D)\right).\]
\end{thm}

\section{Non-analytic preimages and direct images}
\label{sec:NAInverseImage}

In this section, we will finally study non-analytic preimages of direct image sheaves and vice versa. 
For our purpose, the following definition is useful:

\begin{defn}\label{defn:id}
Let $\sF$ be a sheaf on a complex space $X$.
We say that $\sF$ satisfies the property (id) if the following holds:
For any irreducible open set $W\subset X$ and sections $s,t\in \sF(W)$,
the equality $s=t$ on a non-empty open subset of $W$ implies that $s=t$ on $W$.
\end{defn}

Property (id) means that the identity theorem generalizes to sections of $\sF$.
Actually, the identity theorem for irreducible complex spaces (\cf \eg \S\,1.3 in \cite[\Chap 9]{GrauertRemmertCAS}) 
implies that the structure sheaf $\cO_X$ of a complex space satisfies (id).
Moreover, we have:

\begin{lem}\label{lem:id}
Let $X$ be a locally irreducible complex space.
Then a coherent analytic sheaf $\sF$ on $X$ satisfies the property (id) if and only if it is torsion-free.
\end{lem}

\begin{proof}
Let $\sF$ be a torsion-free coherent sheaf on $X$ and $F:=L(\sF)$ the associated linear space.
Then, by \autoref{rem:CharTFWithPC},
a section of $\sF$ is uniquely defined by it values on the locally irreducible primary component $E:=\PC(F)$ of $F$.
\Ie for $W\subset X$ and $s,t \in\sF(W)=\Hom(F_W,W\stimes\CC)$, $s|_E=t|_E$ is equivalent to $s=t$.
So, the desired property follows by the identity theorem applied to $E$.

Conversely, it is clear that sheaves with torsion on a locally irreducible space do not satisfy (id).
\end{proof}

\smallskip
For non-coherent sheaves, the equivalence of \autoref{lem:id} does not hold in general: 
The sheaf $\sC$ of continuous functions on an irreducible complex space $X$ is torsion-free as $\cO_X$-module sheaf,
but it does not satisfy (id).

\medskip
The property (id) is useful in the context of non-analytic preimages:

\begin{lem}\label{lem:preimage}
Let $\pi\colon Y\rightarrow X$ be a proper modification of a locally irreducible complex space $X$,
and $\sF$ a sheaf on $X$ satisfying (id). Then for $U\subset Y$ open:
\begin{eqnarray}\label{eq:presheaf1}
\pi^{-1} \sF (U) = \lim_{\substack{\longrightarrow\\ V\supset\pi(U)}} \sF(V),
\end{eqnarray}
where the limit runs over the open neighborhoods of $\pi(U)$.
\end{lem}

\begin{proof}
As $X$ is locally irreducible, we can assume that $X$ and $Y$ are connected.
Recall that $\pi^{-1} \sF$ is the sheaf associated to the presheaf 
\begin{eqnarray*}
U \mapsto F(U):=\lim_{\substack{\longrightarrow\\ V\supset\pi(U)}} \sF(V)
\end{eqnarray*}
where $U\subset Y$ is open and the limit runs over the open neighborhoods of $\pi(U)$.
We have to show that the presheaf $F$ is canonical (\ie it is already a sheaf).

\medskip
\looseitem{(i)} Existence/Gluing-axiom:
Let $U\subset Y$ be covered by open sets $U_i$, $i\in I$, and let $s_i\in F(U_i)$ satisfy $s_i=s_j$ on $U_{ij}:=U_i\cap U_j$. 
By definition of the inductive limit, 
$s_i\in F(U_i)$ means there are an open set $V_i\supset \pi(U_i)$ and a section $f_i\in\sF(V_i)$ 
with $s_i=[f_i]$ ($s_i$ is represented by $f_i$). 
A priori, we just get $f_i=f_j$ on $\pi(U_{ij}){\subset} V_{ij}$, where $V_{ij}=V_i\cap V_j$. 
Without loss of generality, we can assume that 
each connected component of $V_{ij}$ contains an open subset of $\pi(U_{ij})$ ($\pi$ is a modification). 
So, (id) for $\sF$ implies that $f_i=f_j$ on $V_{ij}$. 
As $\sF$ is a sheaf, there is a section $f\in\sF(V)$ with $f|_{V_i}=f_i$, where $V:=\bigcup_{i\in I} V_i\supset \pi(U)$. 
$f$ represents an $s\in F(U)$ with $s|_{U_i}=s_i$.

\medskip
\looseitem{(ii)} Uniqueness-axiom:
Let $U\subset Y$ be a connected open set, covered by open sets $U_i$, $i\in I$, and let $s,t\in F(U)$ satisfy $s=t$ on $U_{i}$ for all $i\in I$. 
By definition of the inductive limit, there are a connected open set $V\supset \pi(U)$ and sections $f,g\in\sF(V)$ with $s=[f]$ and $t=[g]$. 
We get $f=g$ on $\pi(U_{i})$. Since $\pi(U_i)$ contains an open subset of $V$, 
(id) implies that $f=g$ on $V$. (We have not directly used the uniqueness-axiom for $\sF$ because it is contained in (id)).
\end{proof}

As a special case, we have:

\begin{lem}\label{lem:injection}
Let $\pi\colon Y\rightarrow X$ be a proper modification of a locally irreducible complex space $X$,
and $\sG$ a sheaf on $Y$ satisfying (id). Then for $U\subset Y$ open:
\begin{eqnarray}\label{eq:presheaf}
\pi^{-1}\pi_* \sG (U) = \lim_{\substack{\longrightarrow\\ V\supset\pi(U)}} \sG(\pi^{-1}(V)),
\end{eqnarray}
where the limit runs over the open neighborhoods of $\pi(U)$, and
the canonical homomorphism
$\pi^{-1} \pi_* \sG \rightarrow \sG$ is injective so that $\pi^{-1}\pi_* \sG$ is a subsheaf of $\sG$.
\end{lem}

\begin{proof}
As $X$ is locally irreducible, we can assume that $X$ and $Y$ are connected.
Here, $\pi^{-1}\pi_* \sG$ is the sheaf associated to the presheaf 
\begin{eqnarray*}
U \mapsto F(U):=\lim_{\substack{\longrightarrow\\ V\supset\pi(U)}} \sG(\pi^{-1}(V))
\end{eqnarray*}
where $U\subset Y$ is open and the limit runs over the open neighborhoods of $\pi(U)$.
The canonical homomorphism $\pi^{-1}\pi_* \sG \rightarrow \sG$ is then induced by the restrictions $\sG(\pi^{-1}(V))\rightarrow \sG(U)$.

\medskip
By \autoref{lem:preimage}, $F$ is canonical, \ie \eqref{eq:presheaf} holds.
It is now easy to see that the canonical homomorphism $\psi\colon \pi^{-1}\pi_* \sG \rightarrow \sG$ is injective.
Let $s_x \in \big(\pi^{-1}\pi_* \sG\big)_x$. Then \eqref{eq:presheaf} implies that $s_x$
is represented by a section $s\in \sG(U)$, where $U$ is an open neighborhood of $K_x:=\pi^{-1}(\pi(x))$.
But our assumptions yield that $K_x$ is connected, and so we can assume that $U$ is a connected neighborhood of $K_x$. 
Assume that $\psi(s_x)=0$. This means that $s$ is vanishing on a neighborhood of the point $x$.
But then $s= 0$ as $U$ is connected (and $\sG$ satisfies (id)).
\end{proof}

\spar
\autoref{lem:injection} allows for the following interpretation of $\pi^{-1}\pi_* \sG$:
The sections of $\pi^{-1}\pi_* \sG$ are the sections of $\sG$ which extend along fibers of the modification $\pi$.
This is of particular interest for the choices $\sG=\cO_M$ or $\sG=\Omega^n_M$
when $\pi\colon M\rightarrow X$ is a resolution of singularities,
giving the useful injections $\pi^{-1}\pi_* \cO_M \hookrightarrow \cO_M$ and $\pi^{-1}\pi_* \Omega^n_M \hookrightarrow \Omega^n_M$,
respectively.

\medskip

For the direct image of a non-analytic inverse image sheaf, \autoref{lem:preimage} implies
	\[(\pi_\ast\pi^{-1}\sF)(U)=(\pi^{-1}\sF)(\pi^{-1}(U))=\lim_{\substack{\longrightarrow\\ V\supset U}} \sF(V)=\sF(U),\]
\ie we obtain:

\begin{cor}
Let $\pi\colon Y\rightarrow X$ be a proper modification of a locally irreducible complex space $X$, and $\sF$ a sheaf on $X$ satisfying (id). Then
	\[\pi_*\pi^{-1}\sF=\sF.\]
\end{cor}


\bibliographystyle{amsalpha}
\providecommand{\bysame}{\leavevmode\hbox to3em{\hrulefill}\thinspace}

\end{document}